\documentclass[a4paper,11pt]{article} 
\usepackage[utf8]{inputenc} 

\usepackage{amsmath}
\usepackage{amssymb}
\usepackage{times,commath}

\usepackage{geometry} 
\usepackage{graphicx} 

\usepackage{booktabs} 
\usepackage{array} 
\usepackage{paralist} 
\usepackage{verbatim} 
\usepackage{subfig} 
\usepackage{mathrsfs}

\usepackage{fancyhdr} 
\usepackage{amsthm}
\usepackage{thmtools}
\declaretheorem[name=Theorem,numberwithin=section]{thm}
\declaretheorem[name=Lemma,sibling=thm]{lemma}

\declaretheorem[name=Corollary,sibling=thm]{cor}
\declaretheorem[name=Definition, style=definition,sibling=thm]{defn}
\declaretheorem[name=Remark,style=remark,sibling=thm]{rmk}

\usepackage{nameref,hyperref,cleveref}

\allowdisplaybreaks
\usepackage{lipsum}
\usepackage{enumitem}
\setlist[enumerate]{itemsep=0mm}

\usepackage[nottoc,notlof,notlot]{tocbibind} 
\usepackage[titles,subfigure]{tocloft} 

\newcommand{\A}{\mathcal{A}}
\newcommand{\AU}{\mathcal{A}_\infty(U)} 
\newcommand{\Ak}{\A_\infty^{\oplus k}}
\newcommand{\ana}{\mathcal{C}^\omega} 
\newcommand{\anak}[1]{\mathcal{C}^\omega(#1)^{\oplus k}} 
\newcommand{\vuk}{\anak{V}\cap\Ak(U)} 
\newcommand{\cont}[1]{\mathcal{C}(\overline{#1},\mathbb{R})} 
\newcommand{\contk}[1]{\mathcal{C}(\overline{#1},\mathbb{R})^{\oplus k}} 
\newcommand*\diff{\mathop{}\!\mathrm{d}}   

\newcommand{\mE}{\mathbb{E}}  
\newcommand{\cE}{\mathcal{E}} 
\newcommand{\bbf}{\mathbf{f}} 
\newcommand{\fa}{f^{(\alpha)}} 
\newcommand{\bbfa}{\bbf^{(\alpha)}} 

\newcommand{\G}{\mathcal{G}} 
\newcommand{\bg}{\mathbf{g}}

\newcommand{\HU}{\mathcal{H}(U)} 
\newcommand{\id}{\operatorname{id}}  
\newcommand{\imag}{\mathrm{i}}  
\newcommand{\im}{\operatorname{im}}  
\newcommand{\cL}{\mathcal{L}}  
\newcommand{\M}{\mathcal{M}}

\newcommand{\tnu}{\tilde{\nu}} 
\newcommand{\bone}{\mathbf{1}} 
\newcommand{\ok}{\mathbf{1}^{\oplus k}} 

\newcommand{\om}{\omega}

\newcommand{\mP}{\mathbb{P}} 
\newcommand{\probsp}{(\Omega,\mathcal{F},\mP)} 
\newcommand{\cP}{\mathcal{P}} 

\newcommand{\proj}[1]{\mathbb{R}P^{#1}}  
\newcommand{\posq}{\mathbb{R}^d_+} 
\newcommand{\real}{\operatorname{Re}} 
\newcommand{\cS}{\mathcal{S}}  
\newcommand{\tr}{\operatorname{tr}} 

\newcommand{\bw}{\mathbf{w}} 

\newcommand{\dhat}[1]{\widehat{#1}}  

\newcommand{\oso}{{\sigma\omega'=\omega}}
\newcommand{\cond}{\, \bigm| \,}

\DeclareFontFamily{U}{matha}{\hyphenchar\font45}
\DeclareFontShape{U}{matha}{m}{n}{
	<5> <6> <7> <8> <9> <10> gen * matha
	<10.95> matha10 <12> <14.4> <17.28> <20.74> <24.88> matha12
}{}
\DeclareSymbolFont{matha}{U}{matha}{m}{n}
\DeclareFontSubstitution{U}{matha}{m}{n}

\DeclareFontFamily{U}{mathx}{\hyphenchar\font45}
\DeclareFontShape{U}{mathx}{m}{n}{
	<5> <6> <7> <8> <9> <10>
	<10.95> <12> <14.4> <17.28> <20.74> <24.88>
	mathx10
}{}
\DeclareSymbolFont{mathx}{U}{mathx}{m}{n}
\DeclareFontSubstitution{U}{mathx}{m}{n}

\DeclareMathDelimiter{\vvvert}{0}{matha}{"7E}{mathx}{"17}

\title{Bai-Pollicott Algorithm for Markovian Products of Positive Matrices}
\author{Fan Wang\thanks{Department of Statistics, University of Oxford. Email: \href{mailto:wangfan.ox@gmail.com}{wangfan.ox@gmail.com}}  \and David Steinsaltz\thanks{Department of Statistics, University of Oxford. Email: \href{mailto:steinsal@stats.ox.ac.uk}{steinsal@stats.ox.ac.uk} }}
\date{} 

\begin{document}
	\maketitle
\begin{abstract}
	We consider the problem of estimating the top Lyapunov exponents for Markovian products of positive matrices.
	We define a new transfer operator as a matrix of classical transfer operators and prove its spectral properties. With the spectral properties, we generalize (and provide a more theoretically rigorous foundation for) an algorithm that was introduced informally by Bai \cite{Bai07} and formally by Pollicott \cite{Pollicott} based on dynamical zeta functions.
\end{abstract}

\section{Introduction}
\subsection{Background}
Let $A_1(\om), A_2(\om),\dots$ be a stationary sequence of random matrices taking values in 
$ GL(d,\mathbb{R}) $, satisfying the \emph{finite expectation condition} ${\mE[\sup\{\log\|A(\om)\|,0\}]<\infty}$. Denoting by
\[
{ S_n(\om):=A_n(\om)\dots A_1(\om)},
\] 
Fekete's sub-additivity lemma \cite{Fekete} implies that 
\begin{equation}\label{eq:defn.le}
\gamma:=\lim_{n\to\infty}\dfrac{1}{n}\mE[\log\|S_n(\om)\|]
\end{equation}
exists and takes values in $ \mathbb{R}\cup\{-\infty\} $. Moreover, because matrix norms are equivalent, the limit does not depend on the choice of $ \| \cdot \| $. This limit was called the \emph{Lyapunov exponent} associated with the random matrix products and is of great importance, as it describes the stability of a system.

The asymptotic behaviour of non-commutative products has been studied extensively over the past 60 years. 
The first significant result is the famous 1960 Furstenberg--Kesten Theorem  \cite{Furstenberg1}, stating that, under finite log expectation condition, 
\[
	\lim_{n\to\infty}\dfrac{1}{n}\log\|S_n(\om)\| \ \text{ exists and equals } \gamma, \quad \mathbb{P} \text{-a.s.}
\]
 Furstenberg and Kifer \cite{Furstenberg2,Furstenberg3} showed that when the matrix sequence is i.i.d. and when $ x $ is a non-zero vector in $ \mathbb{R}^d $, under the finite expectation condition and the irreducibility condition
 \begin{equation}\label{eq:Sn.x} 
 	\lim_{n\to\infty}\dfrac{1}{n}\log\|S_n(\om)x\| \ \text{ exists and equals } \gamma, \quad \mathbb{P} \text{-a.s.},
 \end{equation}
and the Lyapunov exponent can be written as
 \[ 
 	\gamma=\iint \log\|A(\om)x\|\mP(\diff\om)\nu(\diff x), 
 \]
 where $ \nu $ is an  invariant measure supported on the projective space $  \mathbb{R}P^{d-1} $ with respect to $ \mP $. That is, $\nu$ satisfies
 \[ 
 	\iint f\left(\dfrac{A(\om)x}{\| A(\om)x \|}\right)\mP(\diff\om)\nu(\diff x)=\int f(x)\nu(\diff x), 
 \]
 for any Borel bounded function $ f $ on the projective space $\mathbb{R}P^{d-1}  $.
 This closed-form representation turns out, unfortunately, not to be of much use for
 actually evaluating the Lyapunov exponent, as an invariant measure can be explicitly computed only
 in exceptional cases \cite{rarecases}. 
 
 In 2010 Pollicott \cite{Pollicott} proposed an alternative algorithm  for computing
 Lyapunov exponents when the matrices are positive and the matrix sequence is i.i.d., based on
 Ruelle's theory of transfer operators \cite{Ruelle} and Grothendieck's classic work on nuclear operators \cite{Grothendieck1}. And this method is proved to be more efficient than classical algorithms such as weak disorder expansion \cite{wk.exp} and cycle expansion \cite{cyc.exp}.
In fact, this algorithm has already appeared earlier in 2007 in an informal form introduced by Bai \cite{Bai07}. In this article we refer to this algorithm as Bai-Pollicott algorithm.
 
 Starting from a probability measure $\mu=\sum_{i=1}^m p_i\delta_{M_i}$
 on a finite set $\{M_1,M_2,\dots, M_m\}$ of positive matrices, he defined an associated
 family of transfer operators $\cL_t$ on a space $\G$ of complex-valued
 functions on an appropriate subset of $\proj{d-1}$ by the expression 
 \begin{equation}\label{eq:firstLt}
 	(\cL_tf)(x)=\sum_{i=1}^m p_i e^{t \log\|M_i x\|}f\left(\dfrac{M_ix}{\|M_i x\|}\right),\quad f\in\G,
 \end{equation}
where $t$ is a real perturbation parameter and $\mE$ is the expectation with respect to 
the random choice of $\omega$. It can be shown that $\cL_t$ possesses an isolated and simple top eigenvalue $\lambda(t)$, and that $\lambda'(0)=\gamma$. This transforms the problem of computing the
Lyapunov exponent into a question about the top eigenvalue of the transfer operator $\cL_t$, a question that is
more amenable to analytic methods. 

After choosing $\G = \A_\infty(U)$ where $\A_\infty(U)$ denotes the set of complex analytic function on a properly chosen domain $U$ with continuous extension to the boundary,  it is proved in \cite{Ruelle} that these $\cL_t$ are trace-class operators, thus it is possible to define the determinant function 
\[
	d(z,t):=\det(\mathcal{I}-z\cL_t),
\]
and so to obtain the top eigenvalue of $\cL_t$ by calculating the largest zero of $d(z,t)$. 
If we expand the analytic function $d(z,t)$ in powers of $z$
\begin{equation}\label{eq:det.fn.exp}
	d(z,t)=1+\sum_{i=1}^\infty a_i(t)z ^i,
\end{equation}
and truncate to the first $p$ terms in \cref{eq:det.fn.exp}, we obtain estimates
for $\lambda'(0)$, hence for the Lyapunov exponent
\begin{equation}\label{eq:Pollicott.est}
	\gamma^{(p)}=\dfrac{\sum_{i=1}^p a_i'(0)}{\sum_{i=1}^p i a_i(0)}.
\end{equation} 
Pollicott shows, in addition, that $\gamma^{(n)}$ converges to the Lyapunov exponent $\gamma$ as $n\to\infty$, and that $|\gamma^{(n)}-\gamma|\sim O(r^{n^{1+1/(d-1)}})$ for some $0<r<1$ \cite{Pollicott,Jenkinson}.

It is worth mentioning that there are a few small gaps 
in Pollicott's original paper (pointed out by Ian Morris or by various anonymous referees): 
\begin{itemize}
\item the choice of the function space $\mathcal{G}$, or more precisely, the existence of the domain $U$ on which $\mathcal{G}$ can be defined, is not clearly stated for the case of dimension $d\ge 3$;
\item the explicit formula for the determinant function (Lemma 4.2 in \cite{Pollicott}) does not hold when $ d \ge 3 $; 
\item because of missing details in the proof of the Perron--Frobenius--Ruelle Theorem in \cite{Pollicott}, concerns have been raised, such as why it is possible to apply Grothendieck trace formula when the chosen space $\G=\A_\infty(U)$ is not known to satisfy the approximation property (that is, informally speaking, the identity map on $\G$ might not be approximated by linear operators of finite rank on some compact subset of $\G$); why the eigenmeasure $\nu$ obtained from the dual operator of the transfer operator acting on the space of bounded analytic functions in the proof of Lemma 3.1 in \cite{Pollicott}, while being a functional on the space of complex analytic functions, restricts to a probability measure; and how to extend the real function $f(x) = \log(\|Ax\|/\|x\|)$ from a real argument to a complex argument in such a way that $f(x)\in \A_\infty(U)$.
\end{itemize}
Moreover, we remark here that our definition of $\cL_t$ given by \cref{eq:firstLt} differs slightly from Pollicott's original in how it specifies the exponential weights.  
The weight functions given in \cite{Pollicott} would force us to average $k$ separate calculations when working with a set of $k$ matrices;
our choice of weight function reduces this to a single calculation. Another advantage of our version is that, from our final formulas \eqref{eq:an0} - \eqref{eq:tr/dtr}, we can see that one only need to calculate the eigenvalues for each finite-length product while in \cite{Pollicott}, computations for the eigenvectors are necessary.

Because of these small errors in Pollicott's paper, some researchers (Ian Morris and various anonymous referees) have gained the impression that his approach is fundamentally flawed.
In the present work we repair these issues in the course of generalizing 
Bai-Pollicott algorithm to the setting of matrix products that are Markovian rather than i.i.d. As the present paper does not directly cite the results of \cite{Pollicott} in our main theorems, but only imitates its methods, one of the contributions of the present paper is to provide a more secure foundation for Bai-Pollicott algorithm, which we also generalize. Most of the results have appeared in the first author's dissertation \cite{FanThesis}.

It is also worth noticing that Ian Morris \cite{Mor22dim, Mor22gle} gives an alternative way of studying transfer operators by considering them acting on the Bergman space $\mathcal{G} = \mathcal{A}^2(U)$ where $\mathcal{A}^2(U)$ is defined to be the set of square-integrable holomorphic functions on $U$ with respect to the Lebesgue measure for a properly chosen domain $U$.

\subsection{Basic definitions and structure of the paper}

We posit a finite set $\{M_i \}_{1\le i \le k}$ of positive matrices in $GL(d,\mathbb{R})$, a $k\times k$ stochastic matrix $P=(p_{ij})_{1\le i, j\le k}$, and an initial probability vector $\mathbf{p_0}$. 
Then on the probability space $\probsp$ of shift over $k$ symbols, we can construct a Markov chain $A(\om)$ of random matrices such that 
\[
	\mP\bigl( A(\sigma\om)=M_j \cond A(\om)=M_i \bigr)=p_{ij},\quad 1\le i,j\le k,
\]
where $\sigma$ denotes the shift map on $\Omega$.
In words, if the current matrix is $M_i$, then we have probability $p_{ij}$ to choose $M_j$ to be the next matrix. In this article, we always assume the transition matrix $P$ is strictly positive. 
We write $A_n(\om):=A(\sigma^n\om), \ n\ge 1$.

Subadditivity implies that the Lyapunov exponent $\gamma$ associated with this problem,
defined by
\[
	\gamma:=\lim_{n\to\infty}\dfrac{1}{n}\mE[\log\|S_n(\om)\|], 
\]
exists, and takes values in $\{-\infty\}\cup\mathbb{R}$, where $S_n(\om):=A_n(\om)A_{n-1}(\om)\cdots A_1(\om)$. The sub-additive ergodic theorem implies that the limit
\[
	\lim_{n\to\infty} \dfrac{1}{n}\log\|A_n(\om)A_{n-1}(\om)\cdots A_1(\om)\|
\]
exists $\mP$-almost surely, and equals the Lyapunov exponent when the finite expectation condition
$\mE[\log^+ \|A(\om)\|]<\infty$ holds. 

The next three sections will lay out the definitions and fundamental properties
for the transfer operators. \Cref{sec:proj.act} aims at proving the existence of
a uniformly contracting domain, stated as \autoref{cor:choice.nbhd.complex}. This is required for
defining the function space $ \mathcal{G} $. \Cref{sec:mkv.to} defines the basic transfer operators, including the description of the relevant
function spaces, leading up to the crucial spectral results, most importantly a Perron--Frobenius--Ruelle theorem, collected in \autoref{thm:spec.M}. These results are then
extended in \Cref{sec:ana.perturb.to} to the parametrised family of transfer operators, with the corresponding results in \autoref{thm:ana.perturb.Lt}. The extension of Bai-Pollicott algorithm can then
be stated and proved in \Cref{sec:alg}.

The present article describes the framework for generalizing Bai-Pollicott algorithm to Markovian sequences of positive matrices.
One might wish for a more full generalization of Bai-Pollicott algorithm, to Markovian sequences of
random invertible (rather than positive) matrices. 
It is not clear what form such a generalization might take.
In the subsequent article \cite{TO}, we consider Markovian products of general invertible matrices,
but in that general setting there is no formula of equivalent simplicity to Bai-Pollicott algorithm.
What is proved there is that, if the matrix set satisfies the strong irreducibility and the contracting property, then after properly choosing the new function space $\mathcal{G}$ defined on $\Omega^+\times\proj{d-1}$ (where $\Omega^+$ denotes the one-sided full shift space)  and defining the corresponding transfer operator in the following form
$$(\cL_g w)(\omega, x)=\sum_{\oso} p(\omega, \omega') e^{g(\omega', M(\omega')\cdot x)}w(\omega', M(\omega')\cdot x), \quad w\in\mathcal{G},$$
the spectral properties still hold. That is, when $g(\omega,x)=-t\log\|M(\omega)x\| \ (t\in\mathbb{C})$ the transfer operator $\cL_g$ has an isolated and simple top eigenvalue $\lambda(t)$, and $\lambda'(0)=\gamma$. 

What we elucidate in the present paper is the crucial role played by the uniformly contracting property of positive matrices (see \autoref{cor:choice.nbhd.complex}) in
turning this general spectral representation into a practical algorithm, as Pollicott did. We rely on defining the transfer operators on complex analytic functions on a specific domain and in proving the nuclearity of the transfer operators \cite{Ruelle}. However, in the more general case considered in \cite{TO}, without the uniformly contracting property,
we have to rely upon the ``contracting-in-average" property, which is insufficient to imply
that the corresponding transfer operators are nuclear. This prevents us from developing an analogue
of Bai-Pollicott algorithm in the more general setting of \cite{TO}. We also remark that, similar to \cite{Pollicott} in which Pollicott commented that the condition of matrices being positive can be weakened to matrices preserving a positive cone, the results in this paper can also be easily adapted to this weaker assumption.

\section{Projective Actions}
\label{sec:proj.act}
In this section, we study the projective actions of positive matrices. The main result is \autoref{cor:choice.nbhd.complex}, which is essential in defining the Banach space on which the transfer operators are acting.

Given an invertible matrix $A\in GL(d,\mathbb{R})$, we write $\widetilde{A}:\proj{d-1}\to\proj{d-1}$
for the projectivized map satisfying $\widetilde{A}\circ \pi = \pi \circ A$,
where $\pi$ is the natural quotient map from $\mathbb{R}^d$ to the real projective space $\proj{d-1}$.  In this section, we always assume the given matrix $A$ is positive, and we investigate the analytic properties of this induced map. 
Denote $\posq = \{(x_1, \dots, x_d): x_i > 0\}$. 
Notice $ A(\posq) \subset \posq$ 
for any positive matrix $A$, so $\widetilde{A}:\Delta \to \Delta$ when restricted to  $\Delta=\pi(\mathbb{R}^d_+)$. 

As the representation of $\proj{d-1}$ as a quotient manifold of $\mathbb{R}^d$ 
is awkward for computation,
it is conventional to choose representatives of $\Delta$ 
in $\mathbb{R}^d_+$ in one of the two following ways:
\begin{enumerate}[label=(\arabic*)]
\item  Identify a point $\hat{x}\in\Delta$ with its unique representative in
$S_+^{d-1}=S^{d-1} \cap \posq$, where $S^{d-1}$ is the unit sphere in $\mathbb{R}^d$.
Given a positive matrix $A$, the induced map $\bar{A}$
\begin{equation}\label{eq:bar.defn}
	\bar{A}:S^{d-1}\to S^{d-1},\qquad \bar{x}\mapsto A\bar{x}/\|A\bar{x}\|_2
\end{equation}
may be identified with $\widetilde{A}|_{\Delta}$. 

\item Given $x=(x_1,\dots, x_d)\in\posq$, since $x_1\not=0$ we may take this as the 
normalisation constant, thus identifying $\proj{d-1}$ with
\begin{equation}\label{eq:defn.R}
R_+^{d-1}:=\{x=(x_1,\dots, x_d)\in\posq:x_1=1 \},
\end{equation}
The map $\widetilde{A}$ on $\Delta$ may be identified with
\begin{equation}\label{eq:defn.hat}
\dhat{A}:R_+^{d-1}\to R_+^{d-1},\qquad x=(1,x_2\dots,x_d)\mapsto Ax/(Ax)_1, 
\end{equation}
where we denote by $(Ax)_i \ (1\le i\le d)$ the $i$-th entry of the vector $(Ax)$.
We can naturally identify $R^{d-1}_+$ with
\[
	\mathbb{R}^{d-1}_+:=\{(x_2, \dots, x_d)\in\mathbb{R}^{d-1}: x_i>0, 2\le i\le d \}. 
\]
\end{enumerate}

We will say a complex analytic function $\psi:\mathbb{C}^d \to \mathbb{C}^d$ maps an open set $U_1\subset \mathbb{C}^d$ \emph{strictly inside} $U_2\subset \mathbb{C}^d$ if $\overline{\psi(U_1)}\subseteq U_2$, where $\overline{\psi(U_1)}$ is the closure of the image of $U_1$. Similarly, a real analytic function $\varphi:\mathbb{R}^d\to\mathbb{R}^d$ maps an open set $V_1\subset \mathbb{R}^d$ strictly inside $V_2\subset \mathbb{R}^d$ if $\overline{\varphi(V_1)}\subseteq V_2$.
 
The real map $\dhat{M}$ induced by the positive matrix $M$ can be extended naturally to a complex map. Let 
 \begin{align*}
 	C^{d-1}&:=\{z=(1, z_2,\dots, z_d):  z_i\in\mathbb{C}, \ 2\le i\le d \},\\
 	C^{d-1}_+&:=\{z=(1, z_2,\dots, z_d): \real z\in\posq \}.
 \end{align*}
 Then we can define $ \dhat{M}: C^{d-1}\to C^{d-1}$ by 
$$
 z=(1, z_2,\dots, z_d)\mapsto \dfrac{Mz}{(Mz)_1}.
$$
Here $(Mz)_1\not=0$ since $\real (Mz)_1=(M\real z)_1\not =0$.
 If $M= \left(a_{ij}\right)_{1\le i,j\le d}$, we may write $\widehat{M}$ explicitly as a projective transformation
 \[
 	(z_2,\dots, z_d)\mapsto \left( \dfrac{a_{21}+\sum_{j=2}^d a_{2j}z_j}{a_{11}+\sum_{j=2}^d a_{1j}z_j},\dots,\dfrac{a_{d1}+\sum_{j=2}^d a_{dj}z_j}{a_{11}+\sum_{j=2}^d a_{1j}z_j} \right).
 \]

\begin{lemma}\label{lm:choice.nbhd.complex}
Let $M$ be a positive matrix, and $\widehat{M}$ be the induced map of $M$ acting on $C^{d-1}$. Then there exists a bounded, convex, and connected open set $U$ in $C_+^{d-1}$ such that $\widehat{M}|_U$ maps $U$ strictly inside $U$, and $\widehat{M}|_U$ is a complex analytic function. Moreover, $U$ can be chosen such that $\widehat{M}|_V$ maps $V\subset R_+^{d-1}$ strictly inside $V$ as a real analytic map, where $V=\real U$ is a connected open set.
\end{lemma}

\begin{proof}
Denote
\begin{equation}\label{eq:U0}
U_0:=\{(z_2,\dots, z_d)\in C_+^{d-1}: \real(z_j/z_i)>0, \ 2\le i, j\le d\}.
\end{equation}
If we write $z_j=x_j+\imag y_j$, where $x_j\in\mathbb{R}_+, y_j\in\mathbb{R}$ for each $2\le j\le d$, then $\real(z_j/z_i)>0$ implies that
\begin{equation}\label{eq:cond}
x_ix_j+y_iy_j> 0, \quad  2\le i, j\le d. 
\end{equation}

We first remark that $\widehat{M}|_{U_0}:U_0\to U_0$ is well-defined as a complex map. In fact, if we write $\widehat{M}$ in the form 
 \[(z_2,\dots, z_d)\mapsto \left( \dfrac{a_{21}+\sum_{j=2}^d a_{2j}z_j}{a_{11}+\sum_{j=2}^d a_{1j}z_j},\dots,\dfrac{a_{d1}+\sum_{j=2}^d a_{dj}z_j}{a_{11}+\sum_{j=2}^d a_{1j}z_j} \right)=:(w_2,\cdots, w_d),\]
 where $M=(a_{ij})_{1\le i, j\le d}$. Note that $w_2,\dots, w_d$, and $w_j/w_i \ (2\le i, j\le d)$ are all of the form
 \[
\dfrac{(a_{\gamma1}+\sum_{j\not=1} a_{\gamma j}x_j)+\imag\sum_{j\not=1} a_{\gamma j}y_j}{(a_{\delta 1}+\sum_{j\not=1} a_{\delta j}x_j)+\imag\sum_{j\not=1} a_{\delta j}y_j}, \quad 1\le \gamma, \delta\le d.
\]
Therefore, checking $\widehat{M}(U_0)\subseteq U_0$ reduces to checking 
\begin{align*}
\real &\left(\dfrac{(a_{\gamma1}+\sum_{j\not=1} a_{\gamma j}x_j)+\imag\sum_{j\not=1} a_{\gamma j}y_j}{(a_{\delta 1}+\sum_{j\not=1} a_{\delta j}x_j)+\imag\sum_{j\not=1} a_{\delta j}y_j}\right) \\
&=\dfrac{(a_{\gamma1}+\sum_{j\not =1} a_{\gamma j}x_j)(a_{\delta 1}+\sum_{j\not=1} a_{\delta j}x_j)+\sum_{j\not=1} a_{\gamma j}y_j\sum_{j\not=1} a_{\delta j}y_j}{(a_{\delta 1}+\sum_{j\not=1} a_{\delta j}x_j)^2+(\sum_{j\not=1} a_{\delta j}y_j)^2}\\
&=\dfrac{a_{\gamma 1}a_{\delta 1}+\sum_{i,j\not=1}a_{\gamma i}a_{\delta j}(x_ix_j+y_iy_j)+\sum_{j\not=1} a_{\delta 1}a_{\gamma j}x_j+\sum_{j\not=1} a_{\gamma 1}a_{\delta j}x_j}{(a_{\delta 1}+\sum_{j\not=1} a_{\delta j}x_j)^2+(\sum_{j\not=1} a_{\delta j}y_j)^2}> 0,
\end{align*}
by  $a_{ij}>0, x_j>0 \ (1\le i,j\le d)$, and \eqref{eq:cond}. 

For each $2\le \gamma\le d$, by \eqref{eq:cond}
\begin{align*}
\left|\dfrac{a_{\gamma1}+\sum_{j=2}^d a_{\gamma j}z_j}{a_{11}+\sum_{j=2}^d a_{1j}z_j}\right|^2
=&\dfrac{(a_{\gamma 1}+\sum_{j\not =1}a_{\gamma j}x_j)^2+(\sum_{j\not=1} a_{\gamma j}y_j)^2}{(a_{11}+\sum_{j\not=1} a_{1j}x_j)^2+(\sum_{j\not=1} a_{1j}y_j)^2}\\
=&\dfrac{a_{\gamma 1}^2+\sum_{j,l\not = 1}a_{\gamma j}a_{\gamma l}(x_jx_l+y_j y_l)+2\sum_{j\not =1} a_{\gamma 1}a_{\gamma j}x_j}{a_{11}^2+\sum_{j,l\not=1}a_{1j}a_{1l}(x_jx_l+y_j y_l)+2\sum_{j\not=1} a_{11}a_{1j}x_j}\\
\le &  \max_{1\le j,l\le d}\left\{\dfrac{a_{\gamma j}a_{\gamma l}}{a_{1j}a_{1l}} \right\}=:C_\gamma.
\end{align*}
Hence $\widehat{M}$ maps $U_0$ inside a bounded open set
\[U=\left\{(w_2,\dots, w_d)\in C_+^{d-1}: |w_\gamma|< C_\gamma +1, \ 2\le \gamma\le d\right\}\cap U_0.\]

Now if we restrict the domain from $U_0$ to $U$, we also have that $\widehat{M}$ maps $U$ inside $U$, and it is strictly inside if $\widehat{M}(U)$ is bounded away from $0$. From the formula for $\real(w_j)$ above we can see in order for $\widehat{M}(U)$ to be close to $0$, we would need for $x_i, y_i$ to become arbitrarily large. 
This is excluded since the modulus is bounded in $U$. Thus $\widehat{M}$ maps $U$ strictly inside $U$. Also note $\real U_0=R_+^{d-1}$ and $\im\widehat{M}|_{R_+^{d-1}}\subset U$. Therefore $\widehat{M}|_{\real U}$ maps $\real U$ strictly inside $\real U$. 

Finally, we prove $U$ and $\real U$ are connected. Define
\[
W = \{(x_2, \cdots, x_d)\in R_+^{d-1}: |x_\gamma| < C_\gamma + 1, \  2 \le \gamma\le d \}.
\]
$W$ is obviously convex and thus path-connected. Now for each point $z = (z_2, \dots, z_d) \in U$, write $z_j = x_j + \imag y_j$ for each $j$, and take $x = (x_2, \dots, x_d)$. It easy to see $x \in W$. Construct $\eta : [0, 1] \to C_+^{d-1}$ with $\eta(t) = (1-t)z + tx$ as a path from $z$ to $x$. Note
\[
\eta(t) = (x_2 + \imag t y_2, \dots, x_d  + \imag ty_d).
\]
Since $x_ix_j + t^2 y_iy_j = (1-t^2) x_ix_j + t^2(x_ix_j+y_iy_j) > 0$, the path $\eta$ completely lies in $U$. Therefore each point $z\in U$ is path-connected (in $U$) to a point in the path-connected set $W$. Thus $U$ is also path-connected, and consequently connected. Connectedness of $\real U$ follows by noticing $\real U = W$.
\end{proof}

\begin{cor}\label{cor:choice.nbhd.complex}
Let $\{M_i\}_{i\in I}$ be a finite set of positive matrices, and $\widehat{M}_i$ be the induced map of $M_i$ on $C^{d-1}$. Then there exists a bounded and connected open set $U\subset C_+^{d-1}$ such that for each $i\in I$, $\widehat{M}_i|_U$ maps $U$ strictly inside $U$ and each $\widehat{M}_{i}|_U$ is a complex analytic function. Moreover, $U$ can be chosen such that $\widehat{M}_{i}|_V$ maps $V\subset R_+^{d-1}$ strictly inside $V$ as a real analytic map, where $V=\real U$ is a connected open set.
\end{cor}
\begin{proof}
By the proof of \autoref{lm:choice.nbhd.complex}, for each $i\in I$, we can find $U_i\subset C_+^{d-1}$ such that $\widehat{M}_i$ maps $U_0$ to $U_0\cap U_i$, where $U_0$ is given by \eqref{eq:U0}. Now take $U=\left(\bigcup_{i\in I}U_i\right)\cap U_0$, then $\widehat{M}_i$ maps $U_0$ strictly inside $U$, and thus maps $U$ strictly inside $U$. Finally, since 
\[
\bigcup_{i\in I}U_i= \{(w_2,\dots, w_d)\in C_+^{d-1} : |w_\gamma| < \max_\delta C_\delta + 1, \ 2\le \gamma \le d\}.
\]
we can show that each point $z \in U$ is path-connected (in $U$) to a point in a path-connected set $W$ as in the proof of  \autoref{lm:choice.nbhd.complex}.
\end{proof}

For the rest of the article, we will call an open connected set in $\mathbb{C}^d$ (or in $\mathbb{R}^d$) a \emph{domain}, and $U$ will be assumed throughout to denote a domain that has the properties proved to be possible in Corollary \ref{cor:choice.nbhd.complex}.

\section{Markovian Transfer Operators}
\label{sec:mkv.to}

By \autoref{cor:choice.nbhd.complex} we can find proper domains $U\subset C_+^{d-1}$ and $V=\real U\subset R_+^{d-1}$ that are mapped strictly inside themselves by each $\widehat{M}_i$, where $\{M_i\}_{1\le i\le k}$ is a finite set of positive matrices. Denote by $\ana(V)$ the set of real analytic functions on $V$, and by $\A_\infty(U)$ the set of complex analytic functions on $U$ with continuous extensions to $\partial U$. $\A_\infty(U)$ equipped with the supremum norm,
$\|f\|:=\sup_{z\in U}|f(z)|=\sup_{z\in \partial U}|f(z)|$ is a Banach space. 
Throughout this section $\|\cdot\|$ will refer to this as the supremum norm.

For each pair $(i,j)$ such that $1\le i,j\le k$, we can define an operator
${\M_{ij}:\A_\infty(U)\to\A_\infty(U)}$ by
\begin{equation}\label{eq:M.ij}
(\M_{ij}w)(z)=p_{ij}w(\widehat{M}_jz).
\end{equation}
We note here that the same formula defines an operator on $\ana(V)$, which we also denote by  $\M_{ij}$. For the rest of this chapter, unless otherwise indicated, whenever we define an operator on $\A_\infty(U)$, we will also be defining an operator on $\ana(V)$ as well as on $\cont{V}$, for which the same notation will be used. For convenience, we denote by $\ana(V) \cap \A_\infty(U)$ the set of complex analytic functions on $U$ which restrict to real analytic functions on $V = \real U$. 

Now denote by $\Ak(U)$ the direct sum of $k$ copies of $\A_\infty(U)$. Let $\mathbf{q}=(q_1,\dots,q_k)$ be the stationary distribution for $P$. For each $\mathbf{w}=(w_1,\dots, w_k)\in \Ak(U)$, define ${\|\mathbf{w}\|:=\sum_{i=1}^k q_i\|w_i\|}$. 

For each $1\le i\le k$, define the operator $\mathcal{E}_i:\A_\infty(U)\to\Ak(U)$ by lifting $w\in\A_\infty(U)$ to an element in $\Ak(U)$ with $i$-th entry $w$ and all other entries $0$; define $\mathcal{P}_i:\Ak(U)\to\A_\infty(U)$ to be the natural projection onto the $i$-th component. 

Define the operator $\M:\Ak(U)\to\Ak(U)$ by 
\begin{equation}\label{eq:to.M.Ak}
\M=\sum_{1\le i, j\le k}\mathcal{E}_i\M_{ij}\mathcal{P}_j,
\end{equation}
where $\M_{ij}$ is given by \eqref{eq:M.ij}. Similarly $\M$ acts on $\anak{V}$ as well as on $\contk{V}$. And we also use the non-standard notation $\anak{V}\cap\Ak(U)$ to denote the set of vectors of complex analytic functions which restrict to real analytic functions.

The main goal of this section is to prove \autoref{thm:spec.M}, which states some key spectral properties of these transfer operators. The following lemma from complex analysis, sometimes referred to as Montel's Theorem \cite[Theorem 14.6]{Rudin.cplx}, will be used repeatedly in the proof of \autoref{thm:spec.M} and in later sections.

\begin{lemma}\label{lm:equicont.cpct}
Let $W$ be a domain of $\mathbb{C}^d$ and $\Lambda$ a 
bounded subset of $\A_\infty(W)$. 
Then for any compact subset $K\subset W$, $\Lambda$ is equicontinuous on $K$. \\
In other words, for any $\epsilon>0$, there exists $\delta>0$ (not depending on $w$) such that  whenever $z, z'\in K$ satisfies $\|z-z'\|<\delta$, we have $|w(z)-w(z')|<\epsilon$, for any $w\in\Lambda$.  In fact, there exist $\delta'>0$ and $C_K>0$  (both not depending on $w$) such that whenever $\|z-z'\|<\delta'$, we have 
\[|w(z)-w(z')|\le C_K\|z-z'\|.\]
Consequently, there exists some $C_K'>0$ (not depending on $w$) such that
\[\left| \dfrac{\partial w}{\partial z_i}(z)\right|\le C_K',\quad 1\le i\le d, \ z\in K.\]
\end{lemma}

Here follows the first main result. We will see that this theorem deals with the case $t=0$ of the more general parametrised transfer operators. We begin with the special case because the proof of the more general \autoref{thm:ana.perturb.Lt} will depend upon being reduced to this case. 

We remark here the conclusions \ref{thm.M.item:eigen} and \ref{thm.M.item:spec} are sufficient for the deduction of the algorithm. The rest of the conclusions are generalisations of Lemma 3.1 (c)(d) in \cite{Pollicott} as byproducts. 

\begin{thm}\label{thm:spec.M}
Let $\M$ be the operator on $\Ak(U)$ defined by \eqref{eq:to.M.Ak}. Then
\begin{enumerate}[label=(\alph*)]
\item\label{thm.M.item:bnd} $\M$ is a bounded operator with $\|\M\|=1$;
\item\label{thm.M.item:eigen}The function $\ok:=\mathbf{1}\oplus\cdots\oplus\mathbf{1}\in\Ak(U)$ is the unique eigenfunction of $\M$ corresponding to the eigenvalue $1$;
\item\label{thm.M.item:spec} The spectrum of $\M$ contains the single isolated eigenvalue $1$ (of algebraic multiplicity $1$),
with the rest of the spectrum contained in a closed subset of $\{z\in\mathbb{C}:|z|<1\}$;
\item\label{thm.M.item:lin.fnl} There exist unique bounded linear functionals $\nu_1,\dots, \nu_k$ on $\A_\infty(U)$, such that $\|\nu_i\|=1$ for each $1\le i\le k$, and such that $\M'\nu=\nu$, where $\M'$ denotes the dual operator of $\M$, $\nu=\oplus_{i=1}^k q_i\nu_i$ and ${\mathbf{q}=(q_1,\dots,q_k)}$ is the stationary distribution for $P$;
\item\label{thm.M.item:convg.cplx}For functions $w_1,\dots, w_k\in\A_\infty(U)$, denoting $\bw=(w_1,\dots,w_k)$, we have
\[
	\M^n\mathbf{w}\to (\nu_1(w_1),\dots, \nu_k(w_k)) \ \text{, \quad as } n\to\infty,
\]
in the topology induced by the supremum norm, where $\nu_1,\dots, \nu_k$ are the linear functionals given in \ref{thm.M.item:lin.fnl}.
\end{enumerate}

Moreover, when $\M$ is regarded as acting on $\contk{V}$, then

\begin{enumerate}[resume, label=(\alph*)]
\item\label{thm.M.item:prob.meas} There exist $k$ unique probability measures $\nu_1^*,\dots, \nu_k^*$ on $\overline{V}$, such that 
\begin{equation}\label{eq:M.inv.real}
\int\M\bbf\diff\nu^*=\int\bbf\diff\nu^*,\quad \bbf=(f_1,\dots, f_k)\in\contk{V},
\end{equation}
where $\nu^*=\oplus_{i=1}^k q_i\nu_i^*$ and $\mathbf{q}=(q_1,\dots,q_k)$ is the stationary distribution for $P$;
\item\label{thm.M.item:convg.real}For any $\mathbf{f}=(f_1,\dots,f_k)\in\contk{V}$,
\[
	\M^n\mathbf{f}\to\left(\int f_1\diff\nu_1^*,\dots, \int f_k\diff\nu_k^*\right)
	\ \text{,\quad as } n\to\infty,
\]
in the topology induced by the supremum norm, where $\nu_1^*,\dots, \nu_k^*$ are the probability measures given in \ref{thm.M.item:prob.meas}.
\end{enumerate}
\end{thm}

\begin{proof}
Because we do not plan to prove the results strictly in order, in the following we separate the proof into parts, for ease of reference.

\begin{enumerate}[wide, label= Part {\arabic*}., ref=part \arabic*]
\item \label{M.pf:bnd}
We start with proving \ref{thm.M.item:bnd}. For any $\bw\in\Ak(U)$, we have
\begin{align*} 
\|\M\bw\|&=\sum_{i=1}^k q_i\Bigl\|\sum_{j=1}^k\M_{ij}w_j\Bigr\|\\
&\le\sum_{i=1}^k q_i \sup_{z\in U}
\Bigl|\sum_{j=1}^k p_{ij}w_j(\widehat{M}_jz)\Bigr|\\
&\le\sum_{i=1}^k \sum_{j=1}^k q_ip_{ij}\|w_j\|\\
&=\|\bw\|.
\end{align*}
Since $\|\M\ok\|=\|\ok\|=1$ this proves that $\|\M\|=1$.

\item \label{M.pf:eigenfn}
In this part, we prove \ref{thm.M.item:eigen}. Note here by proving the uniqueness of eigenfunction, we only prove the geometric simplicity of eigenvalue 1. The proof of algebraic simplicity is left to the proof of \ref{thm.M.item:spec}. 

It is obvious that $\ok$ is an eigenfunction corresponding to the eigenvalue $1$. 
Suppose $\bw$ is another eigenfunction corresponding to the eigenvalue $1$. Then since $\M\bw=\bw$, we have, for each $i\in \{1,\dots,k \}$ and $z\in U$, 
\[
	w_i(z)=\sum_{j=1}^k (\M_{ij}w_j)(z)=\sum_{j=1}^k 	p_{ij}w_j(\widehat{M}_jz). 
\]
Choose a compact set $K$ such that $\bigcup_{i=1}^k\widehat{M}_iU\subset K\subset U$.  Assume $|w_i(z)|$ attains its supremum at $\zeta_i\in\partial U$, note $\widehat{M}_j\zeta_i\in K$, we have
\begin{equation}\label{eq:wi.const}
\|w_i\|=|w_i(\zeta_i)|=\left|\sum_{j=1}^k p_{ij}w_j(\widehat{M}_j \zeta_i)\right|\le\sum_{j\not= i}p_{ij}\|w_j\|+p_{ii}\sup_{z\in K}|w_i(z)|\le \sum_{j=1}^k p_{ij}\|w_j\|.
\end{equation}
Then 
\[\|\bw\|=\sum_{i=1}^k q_i\|w_i\|\le \sum_{i,j=1}^k q_i p_{ij}\|w_j\|=\|\bw\|.\]
Therefore the inequalities in \eqref{eq:wi.const} must be equalities, and the maximum of $w_i$ is attained on $K$. Then by the connectedness of $U$ and the maximum modulus principle, each $w_i$ has to be a constant function. Then $\bw$ is a right eigenvector of $P$ corresponding to the eigenvalue $1$, which has to be $\mathbf{1}^{\oplus k}$, completing the proof of \ref{thm.M.item:eigen}.

\item \label{M.pf:lin.fnl}

In this part, we plan to use Schauder-Tychonoff fixed-point theorem \cite[pp456]{D.Sch} to find $\nu$ introduced in \ref{thm.M.item:lin.fnl} as a fixed point of the dual operator $\M'$ of $\M$. To apply the theorem, we need to construct a convex and compact set which $\M'$ preserves. Define
\[
K = \{\nu\in [\Ak(U)]': \nu(\ok) = 1, \ \|\nu\|\le 1\}.
\]
By Banach-Alaoglu theorem \cite[pp68]{Rudin.fnl}, the unit ball $\{\|\nu\|\le 1\}$ is weak*-compact. And $K$ is weak*-closed in the unit ball thus also weak*-compact. It is not hard to see $K$ is convex, and $\mathcal{M}'$ preserves $K$ since 
\[
(\M'\nu)(\ok) = \nu(\M\ok) = \nu(\ok) = 1,
\]
and for $\bw\in\Ak(U), \ \nu\in K$,
\[
\|(\M'\nu)(\bw)\| = \|\nu(\M\bw)\|\le\|\M\bw\|\le\|\bw\|.
\]
Therefore $K$ has the fixed-point property, that is, there exists 
$\nu\in K$ such that $\M'\nu=\nu$. 

Next we prove the decomposition of $\nu$ in \ref{thm.M.item:lin.fnl}. Let $\bone_j$ be a vector with $\bone$ in the $j$-th entry and $0$ elsewhere, set $r_j = \nu(\bone_j) \in \mathbb{R}$. Obviously $\sum_{j=1}^k r_j =  \sum_{j=1}^k \nu(\bone_j) = \nu(\ok) = 1$.

For $1\le j\le k$, define $\nu_j\in (\A_\infty(U))'$ by $\nu_j(w) = \nu(\bw_j)/r_j$, where $\bw_j$ is a vector with $w$ in the $j$-th entry and $0$ elsewhere. Then it is easy to see $\nu_j(\bone) = 1$ and $\nu = \oplus_{j=1}^k r_j\nu_j$.

From $\nu(\bone_j)=\nu(\M\bone_j)$, we obtain
\[
	r_j = \nu(\bone_j)=\sum_{l=1}^k r_l\nu_l(\M_{lj}\mathbf{1})=\sum_{l=1}^k r_lp_{lj}\nu_l(\mathbf{1})=\sum_{l=1}^k r_lp_{lj}. 
\] 
This shows that $\mathbf{r}=(r_1,\dots,r_k)$ is a left eigenvector of $P$ associated with the eigenvalue $1$. Thus $\mathbf{r}=\mathbf{q}$ and
\begin{equation}\label{eq:M.nu.q}
	\nu(\bw)=\sum_{i=1}^k q_i\nu_i(w_i).
\end{equation}
Thus the existence and decomposition of $\nu$ in \ref{thm.M.item:lin.fnl} are proved. We leave rest of \ref{thm.M.item:lin.fnl} (verifications of properties of $\nu_i$) to the next part.

\item \label{M.pf:convg.cplx}
In this part, the goals are \ref{thm.M.item:convg.cplx} and the rest of \ref{thm.M.item:lin.fnl} (that is, uniqueness of $\nu_i$ and $\|\nu_i\| = 1$). We start with the proof of the convergence in \ref{thm.M.item:convg.cplx}.

Note for $1\le j\le k$ and $\bw=(w_1,\dots, w_k)\in\Ak(U)$, 
\begin{equation}\label{eq:bnd.P}
\|\mathcal{P}_j\bw\|=\|w_j\|\le q_j^{-1}\|\bw\|.
\end{equation}
By \ref{M.pf:bnd} of the proof, we have
\begin{equation} \label{E:q-1}
	\|\mathcal{P}_j\M^n\bw\|\le q_j^{-1}\|\M^n\bw\|\le q_j^{-1}\|\bw\|.
\end{equation}
That is to say, $\{\mathcal{P}_j\M^n\bw\}$ is uniformly bounded by $q_j^{-1}\|\bw\|$.
Applying \autoref{lm:equicont.cpct}, we see that $\{\mathcal{P}_j\M^n\bw\}_{n\ge 1}$ is equicontinuous.
Hence the Arzel\`{a}--Ascoli Theorem implies it is relatively
compact, so for any subsequence $(\M^{n_\alpha}\bw)_\alpha$ we may choose a further subsequence $(\M^{n_{\alpha_l}}\bw)_l$ which is convergent. We write $\bw^*=(w_1^*,\dots, w_k^*)$ as the limit point of $(\M^{n_{\alpha_l}}\bw)_l$ as $l\to\infty$.

Since $\M$ is bounded with $\|\M\|\le 1$, for any $\bw\in\Ak(U)$ we have 
$\|\M^n\bw^*\|\le\|\bw^*\|$ for each $n\ge 1$.
At the same time, for $n\ge 1, l\ge 1$, 
$\|\bw^*\|\le\|\M^{n+n_{\alpha_l}}\bw\|,$
which implies that
$\|\bw^*\|\le \|\M^n\bw^*\|.$
Hence for each $n\ge 1$,
\[
  \|\bw^*\|=\|\M^n\bw^*\|
\]
In particular, when $n=1$ assume $|w_i^*|$ attains its supremum at $\zeta_0^{(i)}\in\partial U$ and $|(\M\bw^*)_i|$ at $\zeta^{(i)}\in\partial U$. Then
\begin{align*}
\|\bw^*\| & =\sum_{i=1}^k q_i |w^*_i(\zeta_0^{(i)})|  = \sum_{i=1}^k q_i |(\M\bw^*)_i(\zeta^{(i)})| \\
& \le \sum_{1\le i,j\le k}q_i p_{ij}|w^*_{j}(\widehat{M}_{j}\zeta^{(i)})|\le \sum_{j=1}^k q_j|w_j^*(\zeta_0^{(j)})|\le\|\bw^*\|.
\end{align*}
Therefore the inequalities above must be equalities, implying that $|w_j^*|$ also attains its supremum 
at $\widehat M_j \zeta^{(i)}$ (as we have assumed that all $p_{ij}>0$). These points are all in the interior of $U$,
and consequently, by the connectedness of $U$ and the maximum modulus principle, $w_i^*(z)$ is a constant function on $U$ for each $1\le i\le k$.

By \ref{M.pf:lin.fnl} we have for any $\bw\in\Ak(U)$,
\begin{equation}\label{eq:M.nu}
\nu(\bw)=\lim_{\ell\to\infty}\nu(\M^{n_{\alpha_\ell}}\bw)=\nu(\bw^*)=\sum_{i=1}^k q_iw_i^*\nu_i(\mathbf{1})=\sum_{i=1}^k q_iw_i^*. 
\end{equation}
Comparing \eqref{eq:M.nu} with \eqref{eq:M.nu.q} gives $w_i^*=\nu_i(w_i)$. 
We have shown then that every subsequence in $\{\M^n \bw \}$ has a further subsequence that converges to the same limit $\bw^*$, hence that
\begin{equation}\label{eq:M.convg.cplx}
\lim_{n\to\infty} \M^n\bw = \bw^* = (\nu_1(w_1),\dots, \nu_k(w_k))
\end{equation}
in the topology induced by supremum norm.
Thus \ref{thm.M.item:convg.cplx} is proved. 

Now we are ready to prove rest of \ref{thm.M.item:lin.fnl}. First of all, the evaluations of each $\nu_i \ (1\le i\le k)$ on $\A_\infty(U)$ are completely determined by  \eqref{eq:M.convg.cplx}, thus each $\nu_i \ (1\le i\le k)$ is unique.

Next notice if we let $\bw_i$ be a vector with $w_i$ in the $i$-th entry and $0$ elsewhere, then
\[|\nu_i(w_i)|=\lim_{n\to\infty} |(\M^n\bw_i)_i|\le \lim_{n\to\infty}q_i^{-1}\|\M^n\bw_i\|\le q_i^{-1}\|\bw_i\|=\|w_i\|.\]
Together with $\nu_i(\mathbf{1})=1$, we have $\|\nu_i\|= 1$ for each $1\le i\le k$. This finishes the proof of \ref{thm.M.item:lin.fnl}.

\item \label{M.pf:spec}
Now we are ready to prove \ref{thm.M.item:spec}. 
Define 
\[
  \Gamma:=\{\bw\in\Ak(U):\|\bw\|\le 1, \nu(\bw)=0 \},\quad \M\Gamma:=\{\M\bw:\bw\in\Gamma \}\subset\Gamma. 
\]
For each $\bw\in\Gamma$, $\nu(\bw)=0$ implies that $\lim_{n\to\infty}\M^n\bw=0$. 
By the same reasoning as in \ref{M.pf:convg.cplx}, $\M\Gamma$ is relatively compact. Thus this convergence of $\|\M^n\bg\|$ is uniform for $\bg\in\M\Gamma\subset\Gamma$. Thus for some $0< r <1$, there exists $N>0$ such that 
$\|\M^N\bg\|\le r<1$ for any $\bg\in\M\Gamma$.
That is to say $\|\M^{N+1}\bw\|\le r <1$ for any $\bw\in\Gamma$. Thus $\|(\M|_\Gamma)^{N+1}\|^{1/(N+1)}\le r^{1/(N+1)}$, i.e., the spectral radius of $\M|_\Gamma$ is strictly smaller than 1. Since $\Gamma$ is the unit ball of a subspace of co-dimension 1, and 1 is an eigenvalue, 
it follows that $\operatorname{Spec}(\M) = \{1\} \cup \operatorname{Spec}(\M|_\Gamma)$, proving \ref{thm.M.item:spec}. The algebraic multiplicity $1$ in \ref{thm.M.item:spec} also follows.

\item \label{M.pf:prob.meas} The goal of this part is to prove \ref{thm.M.item:prob.meas}. We start with the uniqueness of a probability measure satisfying \eqref{eq:M.inv.real}. If $\nu^*$ is any probability measure satisfying \eqref{eq:M.inv.real} then by \eqref{eq:M.convg.cplx}, for any ${\bw\in\anak{V}\cap\Ak(U)}$,
\[\nu^*(\bw)=\lim_{n\to\infty}\nu^*(\M^n\bw)=\nu^*(\nu(\bw))=\nu(\bw). \]
As $\anak{V}\cap \Ak(U)$ is dense in $\contk{V}$ with respect to the topology induced by the norm $|\cdot|_\infty$ on $\contk{V}$, the evaluations of $\nu^*$ are fully determined, therefore a probability measure $\nu^*$ satisfying \eqref{eq:M.inv.real} is unique.

Now we construct $\nu^*$ defined on $\contk{V}$ from $\nu$.
By \eqref{eq:M.nu.q}, $\nu=\oplus_{i=1}^k q_i \nu_i$. 
Denote by $\tnu$ the restriction of $\nu$ to $\vuk$. Now for $\bbf\in\contk{V}$, find a sequence $\bbfa$ in $\vuk$ such that $|\bbfa-\bbf|_\infty\to 0$ as $\alpha\to\infty$ , then define
\begin{equation}\label{eq:M.nu.star.alt.defn}
	\nu^*(\bbf):=\lim_{\alpha\to\infty}\tnu(\bbfa). 
\end{equation}

For the well-definedness of $\nu^*$, if $\bbfa=(\fa_1,\dots,\fa_k)$ is a sequence in $\anak{V}\cap\Ak(U)$ such that $|\bbfa|_\infty\to 0$ as $\alpha\to\infty$, then for each $j \in \{1,\dots,k \}$, 
\[ 
	q_j|(\M^n\bbfa)_j|\le |\M^n\bbfa|_\infty\le |\bbfa|_\infty\to 0, \quad (\alpha\to\infty).
\]
Therefore, by the uniform convergence with respect to $n$ above, for each $1\le j\le k$
\[
	\lim_{\alpha\to\infty}|\nu_j^*(\fa_j)|=\lim_{\alpha\to\infty}\lim_{n\to\infty}|(\M^n\bbfa)_j|=\lim_{n\to\infty}\lim_{\alpha\to\infty}|(\M^n\bbfa)_j|=0.
\]
This shows that the limit given by \eqref{eq:M.nu.star.alt.defn} always exists and does not depend on the choice of the sequence of functions.

Now we verify that $\nu^*$ is a probability measure. If $f\in\cont{V}$ is a non-negative function, for each $i\in \{1,\dots,k \}$ 
let ${\bbfa\in\vuk}$ be a sequence such that each $\bbfa$ has $\fa$ in the $i$-th entry and $0$ elsewhere and that ${|\fa-f|_\infty\to 0}$. Then
\[
	\nu_i^*(f)=\lim_{\alpha\to\infty} (\M^n\bbfa)_i\ge 0. 
\]
Therefore $\nu_i^* \ (1\le i\le k)$ is a positive linear functional on $\cont{V}$ with $\nu_i^*(\mathbf{1})=1$. By the Riesz representation theorem, we know there exist probability measures (which we still denote by $\nu_i^*, \ 1\le i\le k$) on $\overline{V}$ such that
\[ 
  \nu_i^*(f)=\int f\diff\nu_i^*,\quad f\in\cont{V}. 
\]

Finally we verify \eqref{eq:M.inv.real} for $\nu^*$. For any $\bbf\in\contk{V}$ let $\bbfa\in\vuk$ be a sequence satisfying ${|\bbfa-\bbf|_\infty\to 0}$ as $\alpha\to\infty$. Then we have
\begin{align*}
|\nu^*(\bbf)-\nu^*(\M\bbf)|
 \le & |\nu^*(\bbf)-\nu^*(\bbfa)|+|\nu^*(\bbfa)-\nu^*(\M\bbfa)|+|\nu^*(\M\bbfa)-\nu^*(\M\bbf)| \\
\le & |\nu^*(\bbf)-\tnu(\bbfa)|+|\nu(\bbfa)-\nu(\M\bbfa)|+|\tnu(\M\bbfa)-\nu^*(\M\bbf)|\\
\to & \ 0 \quad (\text{as } \alpha\to\infty),
\end{align*}
by applying \eqref{eq:M.nu.star.alt.defn} and noticing 
\[
	|\M\bbfa-\M\bbf|_\infty\le |\bbfa-\bbf|_\infty\to 0, \quad \alpha\to\infty. 
\]
That is to say, $\nu^*(\M\bbf)=\nu^*(\bbf)$ for $\bbf\in\vuk$, proving \ref{thm.M.item:prob.meas}.

\item \label{M.pf:convg.real}
In the end, we prove \ref{thm.M.item:convg.real}. 
For $\bbf=(f_1,\dots, f_k)\in\contk{V}$ and  a sequence $\bbfa=(\fa_1,\dots, \fa_k)\in\vuk$ with $|\bbfa-\bbf|_\infty\to 0$ as $\alpha\to\infty$,  letting $n\to\infty$ in the following inequality
For any $n$ and $\alpha$ we have by \eqref{E:q-1}
\begin{align*}
\Bigl|(\M^n\bbf)_i &- \int f_i\diff\nu_i^* \Bigr|_\infty \\
\le & q_i^{-1}|\bbf-\bbfa|_\infty +\left|(\M^n\bbfa)_i-\int \fa_i\diff\nu_i^*\right|_\infty + 
   \left|\int\fa_i\diff\nu_i^*-\int f_i\diff\nu_i^*\right|_\infty .
\end{align*}
Letting $n$ go to $\infty$ and applying \eqref{eq:M.convg.cplx} we see that
\begin{align*}
\limsup_{n\to\infty}\left|(\M^n\bbf)_i-\int f_i\diff\nu_i^* \right|_\infty \le q_i^{-1}|\bbf-\bbfa|_\infty +\left|\int\fa_i\diff\nu_i^*-\int f_i\diff\nu_i^*\right|
\end{align*}
Then letting $\alpha\to\infty$ implies $|(\M^n\bbf)_i-\nu_i^*(f_i)|_\infty\to 0$ by \eqref{eq:M.nu.star.alt.defn}. 
\end{enumerate}

\end{proof}

At the end of this section, we provide the proof of nuclearity of the transfer operators.\begin{defn}\label{defn:nuclearity}
Let $T: B \to B$ be a bounded operator on Banach space $B$. Then we say $T$ is a \emph{nuclear} operator if there exists a sequence $\{u_n\}_{n\ge 0}$ in the topological dual $B'$ of $B$ and a sequence $\{v_n\}_{n\ge 0}$ in $B$ with $\|u_n\|=\|v_n\| = 1$ for all $n\ge0$, such that for $x\in B$,
\[
Tx = \sum_{n = 0}^\infty \rho_n u_n(x)v_n,
\]
where $\sum_{n} |\rho_n| < \infty$. Moreover, we say $T$ is \emph{nuclear of order zero} if for any $p>0$ we have $\sum_{n} |\rho_n|^p < \infty$.
\end{defn}

By the discussion in \cite{Ruelle}, each $\M_{ij}$ is nuclear of order zero. The idea is to use the space $\mathcal{H}(U)$ of analytic functions, which is known to be a nuclear space \cite[II. pp56]{Grothendieck1}, as a bridge, and then consider the following decomposition of $\M_{ij}$
\[
\mathcal{M}_{ij}: \mathcal{A}_\infty(U)\xrightarrow{\alpha} \mathcal{H}(U)\xrightarrow{\beta} \mathcal{A}_\infty(U),
\]
where $\alpha$ is an injection of $\AU$ into $\HU$, and $\beta$ is the transfer operator $\M_{ij}$ extended to $\HU$ with the same form. It is easy to see that for $h\in\HU$, $\beta(h)\in \AU$ since each matrix action maps $U$ strictly into $U$ according to the choice of $U$ given by \autoref{cor:choice.nbhd.complex}. Since every bounded linear map from a nuclear space into a quasi-complete space (in particular, a Banach space) is nuclear of order zero  \cite[II. pp39, pp61]{Grothendieck1}, we know that $\beta$ is nuclear of order zero, and composing with continuous linear operators preserves nuclearity of order zero \cite[II. pp9]{Grothendieck1}, hence $\M_{ij}$ is nuclear of order zero.

\begin{thm}
\label{thm:M.nuclear0} $\M$ is nuclear of order zero.
\end{thm}
\begin{proof}
By the discussion above, $\M_{ij}$ is nuclear of order zero. Since composition preserves nuclearity of order zero as mentioned above, each $\cE_i\M_{ij}\cP_j$ is nuclear of order zero. By \autoref{defn:nuclearity}, it is easy to see that the property of being nuclear of order zero is closed under linear combinations. Therefore $\M = \sum\mathcal{E}_i\mathcal{M}_{ij} \mathcal{P}_j$, is nuclear of order zero. 
\end{proof}

\section{Analytic Perturbation of Markovian Transfer Operators} \label{sec:ana.perturb.to}

We now define the parametrised transfer operators associated with Markovian products of positive matrices, and describe their spectral properties following the same perturbation method given in \cite[Section V.3, V.4]{Bougerol}.

For each $ i, j \in \{1,\dots,k \}$ and $t\in\mathbb{C}$ define the operator $\cL_{ij,t}:\A_\infty(U)\to\A_\infty(U)$ by
\begin{equation}\label{eq:L.ij}
	(\cL_{ij,t}w)(z)=p_{ij}e^{t\log\left\vvvert M_jz\right\vvvert}w(\widehat{M}_jz),
\end{equation}
and $\cL_t:\Ak(U)\to\Ak(U)$ by 
\begin{equation}\label{eq:to.L.Ak}
	\cL_t=\sum_{1\le i, j\le k}\mathcal{E}_i\cL_{ij,t}\mathcal{P}_j ,
\end{equation}
where $\mathcal{E}_i, \mathcal{P}_j$ are the lifting map and the projection map respectively defined in the previous section.
Here $\left\vvvert\cdot\right\vvvert$ denotes the natural extension of the $\ell_2$-norm on $\mathbb{R}^d$ to $\mathbb{C}^d$ by replacing real variables by complex ones. 
 More specifically, for $z=(z_2,\dots, z_d)\in U$ and a positive matrix $A=(a_{ij})_{1\le i, j\le d}$, 
 \begin{equation}\label{eq:vvvert}
 \log\left\vvvert Az\right\vvvert=\dfrac{1}{2}\log\left[\sum_{i=1}^d\left( a_{i1}+\sum_{j=2}^d a_{ij}z_j\right)^2 \right].
 \end{equation}
The matrix entries are all positive, and the domain $U$ has been chosen to guarantee that the function inside logarithm does not take values in the negative real line or zero. Thus, this function does indeed lie in $\mathcal{A}_\infty(U)$. Moreover, it is not hard to see this definition coincides with $\log\|Az\|$ when $z$ is real, where $\|\cdot\|$ denotes the $\ell_2$-norm.

Note that even in the case when the random product is i.i.d., so that $p_{ij}$ does not depend on $i$, our definition is slightly different from the one originally given by Pollicott in \cite{Pollicott}. The reason to choose this definition is because this is simpler in terms of forms and complexity of computations, and it serves the purpose in the same way as Pollicott's original definition.

Now our plan is to first show that $\{\cL_t\}_{t\in W}$ is an analytic family of bounded operators in \autoref{lm:ana.family}, where $W$ is an open set in $\mathbb{C}$. That is, for every $t_0\in W$ we can find $\epsilon > 0$ and a sequence $\mathcal{S}_n$ of bounded operators such that $\sum_{n= 0}^\infty \epsilon^n \|\mathcal{S}_n\| < \infty$, and for $t\in W$ with $|t - t_0| < \epsilon$, 
\[
\mathcal{L}_t = \sum_{n= 0}^\infty (t - t_0)^n \cS_n.
\]
Then by \cite[VII.6]{D.Sch} or \autoref{lm:ana.perturb} below, the analytic perturbation preserves the spectral properties which we have already proved for $\cL_0=\mathcal{M}$ in \autoref{thm:spec.M}, we can obtain an isolated eigenvalue of maximal modulus $\beta(t)$ for $\cL_t$, which can be used to generate the Lyapunov exponent by \autoref{thm:L.beta.le}.

\begin{lemma}\label{lm:ana.family}
There exists $\gamma > 0$ such that $\{ \cL_t : |t| < \gamma\}$ is an analytic family of bounded operators on $\Ak(U)$.
\end{lemma}
\begin{proof}
We first show $\cL_t$ is a bounded operator on $\Ak(U)$. Note for each $\bw\in\Ak(U)$
$$
  \|\cL_t\bw\| = \sum_{i=1}^k q_i \left \| \sum_{j=1}^k\cL_{ij,t}w_j \right\| 
  	\le\max_j e^{|t|\eta(M_j)} \|\bw\|,
$$
where $\eta(M):=\sup_{z\in U}\left |\log\left\vvvert Mz\right\vvvert \right|<\infty,$ owing to $\log\left\vvvert Mz\right\vvvert$ being holomorphic on the bounded domain $U$.

We now define operator $\cS_{ij, n}$ acting on $\A_\infty(U)$ by 
\[
(\cS_{ij, n}w)(z) = \dfrac{p_{ij}\left(\log\left\vvvert M_j z \right\vvvert\right)^n}{n!}w(\widehat{M}_jz),
\]
and $\cS_n$ on $\Ak(U)$ by
\[
\cS_n = \sum_{1\le i, j \le k} \mathcal{E}_i \cS_{ij, n} \mathcal{P}_j,
\]
and proceed to show that $\cS_n$ satisfies the definition of an analytic family of bounded operators for $\cL_t$. We have assumed $t_0=0$ without loss of generality.
Note that for $w\in \A(U)$, we have
\begin{align*}
\|\cS_{ij, n}w\| \le p_{ij} \dfrac{\eta(M_j)^n}{n!}\|w\|.
\end{align*}
Thus $\cS_{ij, n}$ is a bounded operator on $\A(U)$. Moreover, for $\bw \in \Ak(U)$
$$
\|\cS_n\bw\|  = \sum_{i = 1}^k q_i \left\| \sum_{j=1}^k \cS_{ij, n} w_j\right\| 
 \le \eta_{\max}^n \cdot\dfrac{\|\bw\|}{n!},
$$
where $\eta_{\max} = \max_j \eta(M_j)$. Thus $\|\cS_n\|\le \eta_{\max}^n / n!$
and for all $\epsilon>0$
\[
\sum_{n=0}^\infty \epsilon^n \|\cS_n\| \le e^{\epsilon \eta_{\max}} < \infty.
\]

Using the Taylor series remainder formula we have 
\[
\left \| \cL_t - \sum_{n= 0}^N t^n \cS_n \right\|  \le \dfrac{(|t|\eta_{max})^{N+1}}{(N+1)!} e^{|t|\eta_{max}}.
\]
By choosing $\gamma = 1/ \eta_{max}$, it follows that $\cL_t = \sum_{n= 0}^\infty t^n \cS_n$ on $\{t\in\mathbb{C} : |t| < \gamma\}$.
\end{proof}

Recall in \autoref{thm:spec.M}, since the eigenvalue $1$ of the operator $\mathcal{M} = \cL_0$ is isolated, after choosing a sufficiently small open neighbourhood $W$ of $1$ which does not intersect with the rest of the spectrum, $\M$ can be written in the following form 
\[
\M = \mathcal{N} +\mathcal{Q},
\]
where
\[
\mathcal{N} = \dfrac{1}{2\pi \imag}\int_{\partial W} (z I - \mathcal{M})^{-1}\diff z
\]
is the rank-one Riesz projection onto the unique eigenfunction of $\M$ corresponding to the eigenvalue $1$, and $\mathcal{Q} = \M - \mathcal{N}$ has spectral radius strictly smaller than 1. Moreover, ${\mathcal{N}\mathcal{Q} = \mathcal{Q}\mathcal{N} = 0}$ (see e.g. \cite[VII.3]{D.Sch}).
The following general lemma states that this decomposition is preserved under analytic perturbation. 
Proofs can be found in \cite[VII.6]{D.Sch} or \cite[V.3.2]{Bougerol}.
\begin{lemma}\label{lm:ana.perturb}
Let  $\{T(\xi)\}_{\xi\in W}$ be an analytic family of bounded operators on the Banach space $B$ where $W$ is a neighbourhood of $0$ in $\mathbb{C}$. If for some rank-one operator $N$, we have 
\[
\rho = \limsup_{n\to\infty} \|T(0)^n - N\|^{1/n} < 1.
\]
Then we can find $\epsilon > 0$ such that for $|\xi| < \epsilon$, 
\[
T(\xi) = \lambda(\xi) N(\xi) + Q(\xi)
\]
where 
\begin{enumerate}[label = (\alph*)]
\item $\lambda(\xi)$ is the unique eigenvalue of maximal modulus of $T(\xi)$;
\item $N(\xi)$ is a rank-one projection such that $N(\xi)Q(\xi) = Q(\xi)N(\xi) = 0$;
\item The mappings $\lambda(\cdot), \ N(\cdot)$ and $Q(\cdot)$ are analytic;
\item $|\lambda(\xi)| \ge (2+\rho)/3$, and for each $m\in\mathbb{N}$ there exists $c > 0$ such that for each $n\in\mathbb{N}$, 
\[
\left\|\dfrac{\diff^m}{\diff\xi^m}Q^n(\xi) \right\| \le c \left(\dfrac{1+ 2\rho}{3} \right)^n.
\]
\end{enumerate}
\end{lemma}

\begin{thm}\label{thm:ana.perturb.Lt}
Let $\cL_t:\Ak(U)\to\Ak(U)$ be defined as in \eqref{eq:to.L.Ak} for $t\in\mathbb{C}$. Then when $|t|$ is sufficiently small, we can write
\[
\cL_t = \beta(t)\mathcal{N}(t) + \mathcal{Q}(t),
\]
where $\beta(t)$ is the unique eigenvalue of maximal modulus of $\cL_t$, $\mathcal{N}(t)$ is a rank-one projection, $\mathcal{Q}(t)$ has spectral radius strictly smaller than $\beta(t)$, and $\mathcal{N}(t)\mathcal{Q}(t) = \mathcal{Q}(t)\mathcal{N}(t) = 0$.  $\beta(t), \ \mathcal{N}(t)$ and $\mathcal{Q}(t)$ are analytic with respect to $t$.

When $t = 0$, we have $\beta(0) = 1$ and for $\bbf = (f_1, \dots, f_k) \in \cont{V}$
\[
\mathcal{N}(0)\bbf = \int \bbf\diff\nu^*,
\]
where $\nu^*$ is the invariant measure given by \autoref{thm:spec.M} \ref{thm.M.item:prob.meas}
and $\|\mathcal{Q}'(0)\|$ is bounded.
\end{thm}
\begin{proof}
By \autoref{lm:ana.family}, by choosing $W = \{|t| < \gamma\}$ as a neighbourhood of $0$ in $\mathbb{C}$, we have $\{\cL_t\}_{t\in W}$ is an analytic family of bounded operators on $\Ak(U)$. 

By \autoref{thm:spec.M}, we have the decomposition $\cL_0 = \mathcal{M} = \mathcal{N} + \mathcal{Q}$ with the spectral radius 
\[
\rho(Q) = \lim_n \|\mathcal{Q}^n\|^{1/n}=\lim_n \|\mathcal{M}^n - \mathcal{N}\|^{1/n} < 1.
\]
The conclusion follows then directly from \autoref{lm:ana.perturb}.
\end{proof}

Now we are ready to prove our main result relating $\beta(t)$ to the Lyapunov exponent, which is the foundation for the algorithm that we formulate as \autoref{thm:algorithm.mkv}. Similar theorems for the i.i.d. cases can be found in e.g. \cite{clt}, \cite[V.5.2]{Bougerol}, and \cite{Pollicott}.
\begin{thm}\label{thm:L.beta.le}
Given a Markovian product of positive matrices (that is, specifying a finite matrix set, a $k\times k$ stochastic matrix, and an initial probability vector), let $\cL_t$ be the associated operator on $\Ak(U)$ defined as in \eqref{eq:to.L.Ak} and $\beta(t)$ be its maximal eigenvalue as given by \autoref{thm:ana.perturb.Lt}. Then 
$\beta$ is differentiable at 0, and $\beta'(0)$ equals the Lyapunov exponent corresponding to this random matrix product.
\end{thm}
\begin{proof}
For each $i \in \{1,\dots,k \}$ define $\cL^{(0)}_i:\A_\infty(U) \to\A_\infty(U)$ by
$$
(\cL_i^{(0)} w)(z)=p_i e^{t\log\left\vvvert M_iz\right\vvvert}w(\widehat{M}_iz), 
$$
where $\mathbf{p}_0=(p_1,\dots,p_k)$ is the initial probability vector. Let 
\[\cL^{(0)}:=\sum_{i=1}^k \cL^{(0)}_i\cP_i:\Ak(U)\to\A_\infty(U).\]

Now we have for each $w\in\Ak(U)$,
\begin{align*}
(\cL^{(0)} & \cL_t^n\bw)(z)=\sum_{i_0,i_1,\dots, i_n}(\cL_{i_0}^{(0)}\cL_{i_0i_1}\cL_{i_1i_2}\cdots\cL_{i_{n-1}i_n}w_{i_n})(z) \\
&=\sum_{i_0,i_1,\dots,i_n} p_{i_0}p_{i_0i_1}\cdots p_{i_{n-1}i_n} e^{t\log\left\vvvert M_{i_n}\cdots M_{i_0}z\right\vvvert}w_{i_n}(\widehat{M}_{i_n}\cdots\widehat{M}_{i_0}z).
\end{align*}
Applying this to $\ok(V)$ on $V$ we obtain
$(\cL^{(0)}\cL_t^n\ok)(x)=\mE[e^{t\log\|S_n(\omega) x\|}]$ for $x\in V$, where $S_n(\omega) = M_{i_n}\dots M_{i_0}$ denotes the partial product of the first $n+1$ matrices.

By \autoref{thm:ana.perturb.Lt} we can write $\cL_t=\beta(t)\mathcal{N}(t)+\mathcal{Q}(t)$, 
and $\mathcal{Q}(t)=\cL_t-\beta(t)\mathcal{N}(t)$ has spectral radius strictly smaller than $\beta(t)$ with $\mathcal{Q}(t)\mathcal{N}(t)=\mathcal{N}(t)\mathcal{Q}(t)=0$. Then we have
\begin{equation}\label{eq:decomp.exp.exp}
\mE[e^{t\log\|S_n(\omega)x\|}]=(\beta(t)^n \cL^{(0)}\mathcal{N}(t)\ok+\mathcal{Q}(t)^n\ok)(x).
\end{equation}
Letting $\gamma$ be the Lyapunov exponent associated to this problem, then by \eqref{eq:Sn.x} we have
\begin{align*}
	\gamma & =\lim_{n\to\infty}\dfrac{1}{n}\mE[\log\|S_n(\omega)x\|] \\
	& = \lim_{n\to\infty} \dfrac{1}{n}\frac{\dif\phantom{t}}{\dif t} \left( \beta(t)^n \cL^{(0)}\mathcal{N}(t)\ok+\mathcal{Q}(t)^n\ok \right)(x)\bigg|_{t= 0} \\
	& = \lim_{n\to\infty} \dfrac{1}{n} \left(n\beta(0)^{n-1}\beta'(0)\mathcal{L}^{(0)}\mathcal{N}(0)\ok + \beta(0)^n \mathcal{L}^{(0)}\mathcal{N}'(0)\ok +  n\mathcal{Q}(0)^{n-1}\mathcal{Q}'(0)\ok \right)(x),
\end{align*}
where in turn we apply the definition \eqref{eq:defn.le} of the Lyapunov exponent and equation \eqref{eq:decomp.exp.exp}.
To finish the proof, notice by the result of \autoref{thm:ana.perturb.Lt} that the spectral radius at 0 is $\beta(0) = 1$, with $\mathcal{N}(0)\ok = \ok$. Since $\mathcal{N}(t)$ and $\mathcal{Q}(t)$ are analytic with respect to $t$, we have that both $\mathcal{N}'(0)$ and $\mathcal{Q}'(0)$ are bounded. Moreover, $\mathcal{Q}(0)$ has spectral radius strictly smaller than $\beta(0)=1$, thus $\mathcal{Q}(0)^{n-1}$ decays exponentially in $n$. These conclude that $\gamma = \beta'(0)$.
\end{proof}

Finally we prove that $\cL_t$ is nuclear of order zero. We need this to define the determinant function for $\cL_t$.
\begin{thm}\label{thm:L.nuclear0}
Given a Markovian product of positive matrices (that is, specifying a finite matrix set, a $k\times k$ stochastic matrix, and an initial probability vector), let $\cL_t$ be the associated operator on $\Ak(U)$ defined as in \eqref{eq:to.L.Ak}. Then $\cL_t$ is nuclear of order zero.
\end{thm}

\begin{proof}
The proof is essentially the same as that of \autoref{thm:M.nuclear0} as each  $\cL_{ij, t}$ is nuclear of order zero with the same arguments from \cite{Ruelle}.
\end{proof}

\section{The Trace, the Determinant, and Bai-Pollicott Algorithm}
\label{sec:alg}

By \autoref{thm:L.beta.le} the problem of computing top
Lyapunov exponents has been converted to the computation of $\beta(t)$, the top eigenvalue of the Markovian transfer operator $\cL_t$. Note that by
\autoref{thm:L.nuclear0} $\cL_t$ is a nuclear operator of order zero. Therefore the Fredholm determinant
${\det(I-z\cL_t)}$ is an entire function with respect to $z$ \cite[II. pp16.]{Grothendieck1}, and so the trace-determinant formula
\begin{align}\label{eq:detexp}
	\det(\mathcal{I}-z\cL_t)=\exp\left(-\sum_{n=1}^\infty\dfrac{\tr(\cL_t^n)}{n}z^n\right)
\end{align}
provides an explicit expansion through which we may 
calculate arbitrarily precise approximations via the computation of $\tr(\cL_t^n)$ for each $n\ge 1$.

Note in books such as \cite[Chap V]{GGK}, the approximation property of Banach spaces is required to extend the definitions of traces and determinants from finite rank operators to nuclear operators, and unfortunately $\Ak(U)$ is not known to have the approximation property. However, in \cite[II.1]{Grothendieck1}, a general determinant function is defined for nuclear operators on locally convex spaces with the representation $u=\sum \lambda_n x_n'\otimes x_n$ as the formal series
\[
f(z) = 1 - za_1(u) + z^2a_2(u) - \dots + (-1)^n z^n a_n(u) + \cdots,
\]
where
\[
a_n(u) = \sum_{i_1< \dots < i_n} \lambda_{i_1}\cdots \lambda_{i_n}\det (\langle x_{i_\alpha}, x_{i_\beta}'\rangle)_{1\le\alpha,\beta\le n}.
\]
This definition coincides with the trace-determinant formula \eqref{eq:detexp} when $\tr(u)$ is well-defined. In the following results, we will see $\tr(\mathcal{L}_t)$ is well-defined, and an explicit formula for $\tr(\mathcal{L}_t)$ is given by \autoref{thm:tr.L.Ak}.

\begin{lemma}\label{lemma:tr.TO}
Let $\mathcal{L}$ be an operator on $\AU$ given by
\[
(\mathcal{L}w)(z) = \varphi(z)w(\psi(z)),
\]
where $\varphi(z)\in\AU$ and $\psi(z)$ is a holomorphic function mapping $U\subset \mathbb{C}^d$ strictly inside $U$. Then the trace of $\mathcal{L}$ is well-defined and can be given by
\[
\tr(\mathcal{L}) = \dfrac{\varphi(t)}{\det(I - D\psi(t))},
\]
where $t$ is the unique fixed point of $\mathcal{L}$.
\end{lemma}
\begin{proof}
A proof using Fredholm theory can be found in \cite[Lemma2]{Ruelle}. Alternatively, one can prove that the eigenvalues of $\mathcal{L}$ are exactly
\[
\varphi(t)\gamma_1^{n_1}\cdots\gamma_d^{n_d}, \quad n_1,\dots, n_d\ge 0,
\]
where $t$ is the unique fixed point of $\mathcal{L}$ (whose existence and uniqueness are given by \cite[Lemma1]{Ruelle}) and $\{\gamma_i\}_{1\le i\le d}$ are the eigenvalues (with multiplicities) of $D\psi(t)$. (Proof of this result can be found in \cite{May2} when $d = 1$, and in \cite[Theorem 5.5.4]{FanThesis} for the general case.) By the $2/3$-nuclear theorem \cite[II. pp18]{Grothendieck1}, the sum of the eigenvalues equals the trace for $\mathcal{L}$ since it is nuclear operator of order zero. It is easy to see that those eigenvalues sum up to the given formula.
\end{proof}

\begin{lemma}\label{lemma:tr.L.ij}
Let $\cL$ be a nuclear operator on $\A_\infty(U)$. Then 
\[
	\tr(\cE_i\cL\cP_j)=\delta_{ij}\tr(\cL),
\]
where $\delta_{ij}$ is the Kronecker symbol, taking the value $1$ when $i=j$ and $0$ otherwise.
\end{lemma}
\begin{proof}
Assume $\cL$ has the Schmidt representation $\cL=\sum_{n=1}^\infty \rho_n u_n\otimes f_n$, with the trace given by
\[
	\tr(\cL)=\sum_{n=1}^\infty \rho_n f_n(u_n),
\]
by Fredholm theory \cite[II.1]{Grothendieck1}.
Then 
\[
	\cE_i\cL\cP_j=\sum_{n=1}^\infty \rho_n (\cE_i u_n)\otimes (\cP_j^*f_n),
\]
and Fredholm theory also implies
\[
	\tr(\cE_i\cL\cP_j)=\sum_{n=1}^\infty \rho_n(\cP_j^*f_n)(\cE_i u_n)=\sum_{n=1}^\infty \rho_n f_n(\cP_j\cE_i u_n)=\delta_{ij}\tr(\cL).
\]
\end{proof}

\begin{thm}\label{thm:tr.L.Ak}
Let $\cL_t$ be the operator on $\Ak(U)$ given by \eqref{eq:to.L.Ak}, then
\begin{equation}\label{eq:tr.L}
\tr(\cL_t)=\sum_{i=1}^k\tr(\cL_{ii,t})= \sum_{i=1}^k \dfrac{p_{ii}\lambda_{i}^t}{\det(I-D\widehat{M}_i(s_i))}.
\end{equation}
Moreover, 
\begin{equation}\label{eq:tr.L.n}
\tr(\cL_t^n)=\sum_{|\underline{i}|=n}\tr(\cL_{i_1i_2,t}\cdots\cL_{i_{n-1}i_n,t}\cL_{i_ni_1,t})=\sum_{|\underline{i}|=n}\dfrac{p^*_{\underline{i}}\lambda^t_{\underline{i}}}{\det(I-D\widehat{S}_{\underline{i}}(s_{\underline{i}}))},
\end{equation}
where for each $n$-length sequence $\underline{i}=(i_1,\dots, i_n)$, $p^*_{\underline{i}}$ denotes the \textbf{cyclic} probability $p_{i_1i_2}\cdots p_{i_{n-1}i_n}p_{i_ni_1}$, $\widehat{S}_{\underline{i}}$ denotes the product $\widehat{M}_{i_n}\cdots\widehat{M}_{i_1}$, $s_{\underline{i}}$ denotes the unique fixed point of $\widehat{S}_{\underline{i}}$ and $\lambda_{\underline{i}}$ denotes the top eigenvalue of $S_{\underline{i}}:=M_{i_n}\cdots M_{i_1}$.
\end{thm}

\begin{proof}
The first equality of \eqref{eq:tr.L} follows from the definition \eqref{eq:to.L.Ak} and \autoref{lemma:tr.L.ij}. And the second follows from the trace formula \autoref{lemma:tr.TO}.

For \eqref{eq:tr.L.n}, we notice
\[\cL_t^n=\left(\sum_{i,j=1}^k\cE_i\cL_{ij,t}\cP_j\right)^n=\sum_{i_0,i_1,\dots,i_n}\cE_{i_0}\cL_{i_0i_1,t}\cdots\cL_{i_{n-1}i_n,t}\cP_{i_n}, \]
since $\cP_j\cE_i=\delta_{ij}\id_{\A_\infty(U)}$. Therefore by \autoref{lemma:tr.L.ij},
\[\tr(\cL_t^n)=\sum_{i_1,\dots, i_n}\tr(\cL_{i_ni_1,t}\cL_{i_1i_2,t}\cdots\cL_{i_{n-1}i_n,t}).\]
Note 
\[(\cL_{i_ni_1,t}\cL_{i_1i_2,t}\cdots\cL_{i_{n-1}i_n,t}w)(x)=p_{\underline{i}}^*e^{t\log\|S_{\underline{i}}x\|}w(\widehat{S}_{\underline{i}}x),\]
thus the rest follows from trace formula \autoref{lemma:tr.TO}.
\end{proof}

Having established the spectral properties of the Markovian transfer operators and the trace formula, we can follow the same approach as described in \cite{Pollicott} to obtain the estimate \eqref{eq:le.estimate}. By applying \eqref{eq:tr.L.n} for the explicit trace computation and applying \autoref{lemma:modified} for the explicit denominator computation, formulas \eqref{eq:an0}--\eqref{eq:tr/dtr} follow. 
It is not hard to see that these formulas do not require that any eigenvectors be computed.
\begin{thm}\label{thm:algorithm.mkv}
Given a finite set $\{M_i\}_{1\le i\le k}$ of positive and invertible $d\times d$ matrices, a $k\times k$ positive stochastic matrix $P=(p_{ij})_{1\le i,j\le k}$ and an initial probability vector $\mathbf{p}_0$. On a probability space $\probsp$, define a Markov chain $A(\om)$ of random matrices with 
\[
{\mP \bigl(A(\sigma\om)=M_j|A(\om)=M_i \bigr)=p_{ij}}
\]
for $1\le i, j\le k$. Let $\gamma$ be the Lyapunov exponent associated with this random matrix product problem. Then we have 
\[
\gamma=\dfrac{\sum_{n=1}^\infty a_n'(0)}{\sum_{n=1}^\infty n a_n(0)},
\]
where for each $n\ge 1$, $a_n(t)$ is given by 
\begin{align} \label{eq:an}
a_n(t)=\sum_{1\le l\le n}\dfrac{(-1)^l}{l!}\sum_{n_1+\cdots+n_l=n}\dfrac{\tr(\cL_t^{n_1})\cdots\tr(\cL_t^{n_l})}{n_1\cdots n_l},
\end{align}
and $\tr(\cL_t^m)$ is given by \autoref{thm:tr.L.Ak}, for each $m\ge 1$. Moreover, the $\alpha$-th estimate of $\gamma$ can be given by
\begin{equation}\label{eq:le.estimate}
\gamma^{(\alpha)}=\dfrac{\sum_{n=1}^\alpha a_n'(0)}{\sum_{n=1}^\alpha na_n(0)},
\end{equation}
where
\begin{align}
\label{eq:an0}
a_n(0) & =\sum_{1\le l \le n} \dfrac{(-1)^l}{l!}\sum_{n_1+\dots+n_l = n} \tau(n_1)\cdots \tau(n_l), \\
\label{eq:anp0}
a'_n(0) & = \sum_{1\le l \le n} \dfrac{(-1)^l}{l!}\sum_{n_1+\dots+n_l = n} \tau(n_1)\cdots \tau(n_l) \left(\delta(n_1)+\cdots +\delta(n_l)\right), \\
\label{eq:tr0}
\tau(k) & = \dfrac{1}{k} \sum_{|\underline{i}| = k}\dfrac{p_{\underline{i}}^*}{\prod_{j=2}^d \left(1- \lambda_{\underline{i}, j}/\lambda_{\underline{i}, 1}\right)}, \\
\label{eq:tr/dtr}
\delta(k) & = \left (\sum_{|\underline{i}| = k} \dfrac{p_{\underline{i}}^* \log(\lambda_{\underline{i}, 1})}{\prod_{j=2}^d \left(1- \lambda_{\underline{i}, j}/\lambda_{\underline{i}, 1}\right)} \right) \bigg/ \left (\sum_{|\underline{i}|= k} \dfrac{p_{\underline{i}}^*}{\prod_{j=2}^d \left(1- \lambda_{\underline{i}, j}/\lambda_{\underline{i}, 1}\right)} \right),
\end{align}
for each sequence $\underline{i} = (i_1, \dots, i_n)$, $\lambda_{\underline{i}, 1}$ denotes the top eigenvalue of the product $S_{\underline{i}} = M_{i_n}\dots M_{i_1}$ (it is well-defined as $S_{\underline{i}}$ is a positive matrix) and $\lambda_{\underline{i}, j} \ (2\le j \le d)$ denotes the other eigenvalues of $S_{\underline{i}}$, and $p_{\underline{i}}^*$ denotes the cyclic probability $p_{i_1i_2}\cdots p_{i_{n-1}i_n} p_{i_ni_1}$.
\end{thm}
\begin{proof}
Write $f(t, z) = \det(I- z\cL_t)$, then the maximal eigenvalue $\beta(t)$ satisfies $f(t, 1/\beta(t)) \equiv 0$. Here as $\beta(0) = 1$ by \autoref{thm:L.beta.le}, $\beta(t) \not = 0$ when $t$ is sufficiently close to $0$. Also $\beta(t)$ is differentiable at $0$ by \autoref{thm:L.beta.le}, therefore
\[
\beta'(0) = \dfrac{\frac{\partial}{\partial t}|_{t=0}f(t, 1)}{\frac{\partial}{\partial z}|_{z=1}f(0, z)}.
\]
Note since $z=1$ is a simple zero for $f(0, z)$ by \autoref{thm:spec.M}, the denominator cannot be zero. 

Since $f(t, z)$ is an entire function with respect to $z$, we can expand $f(t, z) = \sum_{n= 0} a_n(t) z^n$. And by \autoref{thm:L.beta.le}, the Lyapunov exponent can be given by
\[
\gamma = \beta'(0) = \dfrac{\sum_{n=1}^\infty a_n'(0)}{\sum_{n=1}^\infty na_n(0)},
\]
where the formula \eqref{eq:an} can be obtained by expanding \eqref{eq:detexp}. Finally, denote by 
\begin{align*}
\tau(k) & = \dfrac{1}{k}\tr(\cL_0^k), \\
\delta(k) & = \dfrac{\frac{\diff}{\diff t}|_{t=0}\tr(\cL_t^k)}{\tr(\cL_0^k)},
\end{align*}
then the rest follows from \autoref{thm:tr.L.Ak} and \autoref{lemma:modified}.
\end{proof}

\begin{rmk}
This algorithm contains Bai-Pollicott algorithm as a special case. Similar to Bai-Pollicott algorithm, the super-exponential convergence of the estimation, that is, 
\[
|\gamma^{(n)}-\gamma|=O(r^{n^{1+1/(d-1)}})
\] 
for some $0< r <1$, also follows from the fact that the transfer operator is nuclear of order zero (cf. \cite{Jenkinson} \cite[Section II.1]{Grothendieck1}).
\end{rmk}

\begin{rmk}
An area for further exploration is to generalize the results in this section to Markovian shifts over countable (or uncountable) symbols. Moreover, in the definition \eqref{eq:to.L.Ak}, we construct $\cL$ to be a $k\times k$ matrix of simple transfer operators. In fact, matrices of operators are intensively studied in operator K-theory (see e.g. \cite{wegge1993k}, \cite{murphy2014c})., and $\cL_t$ and $\M$ are in fact homotopic. We expect that there are deeper connections to be drawn between these two areas.
\end{rmk}

\section{Modified Determinant Lemma}
\label{sec:mod.lm}

The goal of this section is to provide a general formula to compute the denominator 
${\det(I-D\widehat{S}_{\underline{i}}(s_{\underline{i}}))}$ of \eqref{eq:tr.L.n}, where for each sequence ${\underline{i}=(i_1,i_2,\dots,i_n)}$, $\widehat{S}_{\underline{i}}$ denotes $\widehat{M}_{i_n}\widehat{M}_{i_{n-1}}\cdots \widehat{M}_{i_1}$ and $s_{\underline{i}}$ denotes the fixed point of $\widehat{S}_{\underline{i}}$.

In \cite{Pollicott}, Lemma 4.2 states the following identity
\[\det(I-D\widehat{S}_{\underline{i}}(s_{\underline{i}}))=\left(1-\dfrac{\det S_{\underline{i}}}{\lambda_{\underline{i}}^2}\right),\]
where $\lambda_{\underline{i}}$ denotes the top eigenvalue of $S_{\underline{i}}:=M_{i_n}M_{i_{n-1}}\cdots M_{i_1}$.
However, we will see that this equation is valid only for $2\times 2$ 
matrices with determinant 1. A generalized formula that holds for all positive matrices is given here.
\begin{lemma} (Modified Determinant Lemma)
\label{lemma:modified}
For any $d\times d$ positive matrix $A$, we have
\[\det(I-D\widehat{A}(x))=\prod_{i=2}^d \left(1-\dfrac{\lambda_i}{\lambda_1}\right),\]
where $\lambda_1$ and $x$ denote the maximal simple positive eigenvalue and the associated eigenvector of $A$, and $\lambda_2,\cdots, \lambda_d$ are the other eigenvalues. 
\end{lemma}
\begin{proof} Let $(1 ~~ t_2 ~~ \cdots ~~ t_d)^T$ be a coordinate system around the point $\bar{x}=(1 ~~ x_2 ~~ \cdots ~~ x_d)^T$, write $\sim$ for the equivalence relation on $\mathbb{R}^d\setminus\{0\}$ given by $\alpha\sim \beta$ if and only if there exists $\lambda\not =0$ such that $\alpha=\lambda \beta$. We have
$$\begin{pmatrix}
  a_{11} & a_{12} & \cdots & a_{1d} \\
  a_{21} & a_{22} & \cdots & a_{2d} \\
 \vdots  & \vdots & \ddots & \vdots \\
  a_{d1} & a_{d2} & \cdots & a_{dd}
 \end{pmatrix}
\begin{pmatrix}
1\\t_2\\ \vdots \\ t_d
\end{pmatrix}=
\begin{pmatrix}
a_{11}+\sum_{j=2}^d a_{1j}t_j\\a_{21}+\sum_{j=2}^d a_{2j}t_j\\ \vdots \\a_{d1}+\sum_{j=2}^d a_{dj}t_j
\end{pmatrix}
\sim \begin{pmatrix}
1\\b_2\\ \vdots \\ b_d
\end{pmatrix} 
$$
where $b_k=\dfrac{a_{k1}+\sum_{j=2}^d a_{kj}t_j}{a_{11}+\sum_{j=2}^d a_{1j}t_j},\quad k=2,3,\cdots, d.$
Therefore, $$\dfrac{\partial b_k}{\partial t_i}=\dfrac{a_{ki}(a_{11}+\sum_{j=2}^d a_{1j}t_j)-a_{1i}(a_{k1}+\sum_{j=2}^d a_{kj}t_j)}{(a_{11}+\sum_{j=2}^d a_{1j}t_j)^2}.$$
At the point $\bar{x}$, the derivative is $\dfrac{\partial b_k}{\partial t_i}\bigg | _{\bar{x}}=\dfrac{a_{ki}-a_{1i}x_k}{\lambda_1}, \quad k,i=2, 3,\cdots, d.$
$$(D\widehat{A}(x))_{i,j}=\dfrac{a_{i+1,j+1}-a_{1,j+1}x_{i+1}}{\lambda_1},\quad i,j=1, 2,\cdots, d-1.$$
So, $$\det(I-D\widehat{A}(x))=\lambda_1^{-(d-1)}\det
\begin{pmatrix}
 \lambda_1- a_{22}+a_{12}x_2 & -a_{23}+a_{13}x_2 & \cdots & -a_{2d}+a_{1d}x_2 \\
  -a_{32}+a_{12}x_3 & \lambda_1-a_{33}+a_{13}x_3 & \cdots & -a_{3d}+a_{1d}x_3 \\
 \vdots  & \vdots & \ddots & \vdots \\
  -a_{d2}+a_{12}x_d & -a_{d3}+a_{13}x_d & \cdots & \lambda_1-a_{dd}+a_{1d}x_d
 \end{pmatrix}$$
Since $a_{11}+\sum_{j=2}^d a_{1j}x_j=\lambda_1$, we have 
\begin{align*}
\prod_{i=1}^d(\lambda-\lambda_i) &=\det
\begin{pmatrix}
 \lambda- a_{11} & -a_{12} & \cdots & -a_{1d} \\
  -a_{21} & \lambda-a_{22} & \cdots & -a_{2d} \\
 \vdots  & \vdots & \ddots & \vdots \\
  -a_{d1} & -a_{d2} & \cdots & \lambda-a_{dd}
 \end{pmatrix}\\
& =\det\begin{pmatrix}
  \lambda-\lambda_1+\sum a_{1j}x_j & -a_{12} & \cdots &- a_{1d} \\
  -a_{21} & \lambda-a_{22} & \cdots & -a_{2d} \\
 \vdots  & \vdots & \ddots & \vdots \\
  -a_{d1} & -a_{d2} & \cdots & \lambda-a_{dd}
 \end{pmatrix}\\
&=\det\begin{pmatrix}
  \lambda-\lambda_1 & -a_{12} & \cdots &- a_{1d} \\
  -a_{21}+\lambda x_2-\sum a_{2j}x_j & \lambda-a_{22} & \cdots & -a_{2d} \\
 \vdots  & \vdots & \ddots & \vdots \\
  -a_{d1}+\lambda x_d-\sum a_{dj}x_j & -a_{d2} & \cdots & \lambda-a_{dd}
 \end{pmatrix}\\
&=\det\begin{pmatrix}
  \lambda-\lambda_1 & -a_{12} & \cdots &- a_{1d} \\
  (\lambda-\lambda_1)x_2 & \lambda-a_{22} & \cdots & -a_{2d} \\
 \vdots  & \vdots & \ddots & \vdots \\
  (\lambda-\lambda_1)x_d & -a_{d2} & \cdots & \lambda-a_{dd}
 \end{pmatrix}\\
&=(\lambda-\lambda_1)^d\det\begin{pmatrix}
  1 & -a_{12} & \cdots &- a_{1d} \\
  x_2 & \lambda-a_{22} & \cdots & -a_{2d} \\
 \vdots  & \vdots & \ddots & \vdots \\
  x_d & -a_{d2} & \cdots & \lambda-a_{dd}
 \end{pmatrix}\\
&=(\lambda-\lambda_1)^d\det\begin{pmatrix}
  1 & -a_{12} & \cdots &- a_{1d} \\
  0 & \lambda-a_{22}+a_{12}x_2 & \cdots & -a_{2d}+a_{1d}x_2 \\
 \vdots  & \vdots & \ddots & \vdots \\
  0 & -a_{d2}+a_{12}x_d & \cdots & \lambda-a_{dd}+a_{1d}x_d
 \end{pmatrix}
\end{align*}
where the third equality follows by multiplying the $i$-th column by $x_i$ and adding it to the first column for each $2\le i\le d$, the fourth follows by noticing $a_{k1}+\sum_{j=2}^d a_{kj}x_j=\lambda_1x_k$ for each $2\le k\le d$, and the last  by multiplying the first row by $-x_i$ and adding it to the $i$-th row for each $2\le i\le d$.

Therefore 
\begin{align*}
\det(I-D\widehat{A}(x))&=\dfrac{1}{\lambda_1^{d-1}}\dfrac{(\lambda-\lambda_1)(\lambda-\lambda_2)\cdots(\lambda-\lambda_d)}{\lambda-\lambda_1}\bigg |_{\lambda=\lambda_1}\\
&=\prod_{i=2}^d\left(1-\dfrac{\lambda_i}{\lambda_1}\right).
\end{align*}
\end{proof}
When $d=2$, since $\lambda_1\lambda_2=\det(A)=1$, we have
$$\det\left(I-D\widehat{A}(x)\right)=1-\dfrac{\lambda_2}{\lambda_1},$$
which coincides with Pollicott's formula. 

\section*{Acknowledgement}
We are grateful to Ian Morris and Christina Goldschmidt for their valuable comments and suggestions as examiners of the first author's PhD thesis \cite{FanThesis}, from which this article is derived. Their feedback helped us to considerably improve the quality and clarity of the presentation. We also appreciate the careful and constructive reports of the anonymous referees, who pointed out several errors and provided useful insights. 

\bibliographystyle{plain}
\bibliography{random_matrices,operator_algebras}

\end{document}